\documentclass[11pt]{article}
\usepackage[english]{babel}
\usepackage{graphics}
\usepackage{latexsym}
\usepackage{amsmath}
\usepackage{amsthm}
\usepackage{amsfonts}
\usepackage{amssymb}
\usepackage{mathrsfs}
\usepackage{stmaryrd}

\newtheorem{thm}{Theorem}[section]
\newtheorem*{thma}{Theorem A}

\newtheorem{lem}[thm]{Lemma}

\newtheorem{quest}{Question}

\theoremstyle{definition}

\newtheorem{defn}{Definition}[section]

\theoremstyle{open}

\newcommand{\PP}{{\mathbb P}}

\newcommand{\C}{{\mathbb C}}

\newcommand{\R}{{\mathbb R}}
\newcommand{\T}{{\mathbb T}}

\newcommand{\Z}{{\mathbb Z}}
\newcommand{\D}{{\mathbb D}}

\newcommand{\DD}{{\mathcal D}}

\newcommand{\diam}{{\textup{diam}}}
\newcommand{\ra}{{\rightarrow}}

\newcommand{\Int}{{\textup{Int}}}

\newcommand{\Cl}{{\textup{Cl}}}

\newcommand{\diff}{\textup{Diff}}
\newcommand{\homeo}{\textup{Homeo}}

\newcommand{\SL}{\textup{SL}}
\newcommand{\SO}{\textup{SO}}

\newcommand{\Mob}{\textup{M\"ob}}

\newcommand{\vol}{\textup{Vol}}

\begin{document}
\title{Topological Entropy and Diffeomorphisms of Surfaces with Wandering Domains}

\date{}

\author{Ferry Kwakkel \thanks{The first author was supported by Marie Curie grant MRTN-CT-2006-035651 (CODY).} 
\and Vladimir Markovic}

\maketitle

\begin{abstract}
Let $M$ be a closed surface and $f$ a diffeomorphism of $M$. A diffeomorphism is said to permute a dense collection
of domains, if the union of the domains are dense and the iterates of any one domain are mutually disjoint. In this note,
we show that if $f \in \diff^{1+\alpha}(M)$, with $\alpha>0$, and permutes a dense collection of domains with bounded geometry, 
then $f$ has zero topological entropy. 
\end{abstract}

\section{Definitions and statement of results}\label{wandering_defn_statements}

A result of A. Norton and D. Sullivan~\cite{NS} states that a diffeomorphism $f \in \diff^3_0(\T^2)$ having {\em Denjoy-type} can not have a wandering disk whose iterates have the same {\em generic shape}. By diffeomorphisms of Denjoy-type are meant diffeomorphisms of the two-torus, isotopic to the identity, that are obtained as an extension of an irrational translation of the torus, for which the semi-conjugacy has countably many non-trivial fibers. If these fibers have non-empty interior, then the corresponding diffeomorphism has a wandering disk. Further, by generic shape is meant that the only elements of $\SL(2,\Z)$ preserving the shape are elements of $\SO(2,\Z)$, such as round disks and squares. In a similar spirit, C. Bonatti, J.M. Gambaudo, J.M. Lion and C. Tresser in~\cite{bonatti} show that certain infinitely renormalizable diffeomorphisms of the two-disk that are sufficiently smooth, can not have wandering domains if these domains have a certain boundedness of geometry. 

In this note, we study an analogous problem, namely the interplay between the geometry of iterates of domains under a diffeomorphism and its topological entropy. To state the precise result, we first need some definitions. Let $(M,g)$ be a closed surface, that is, a smooth, closed, oriented Riemannian two-manifold, equipped with the canonical metric $g$ induced from the standard conformal metric of the universal cover $\PP^1, \C$ or $\D^2$. We denote by $d(\cdot, \cdot)$ the distance function relative to the metric $g$. Let $\diff^r(M)$ be the group of diffeomorphisms of $M$, where for $r \geq 0$ finite, $f$ is said to be of class $C^r$ if $f$ is continuously differentiable up to order $[r]$ and the $[r]$-th derivative is $(r)$-H\"older, with $[r]$ and $(r)$ the integral and fractional part of $r$ respectively. We identity $\diff^0 (M)$ with $\homeo(M)$, the group of homeomorphisms of $M$. 

Given $f \in \homeo(M)$, for each $n \geq 1$, define the metric $d_n$ on $M$ given by $d_n(x,y) = \max_{1 \leq i \leq n}\{d(f^i(x), f^i(y))\}$. 
Given $\epsilon>0$, a subset $U \subset M$ is said to be $(n,\epsilon)$ separated if $d_n(x,y) \geq \epsilon$ for every $x, y \in U$ with $x \neq y$.
Let $N(n,\epsilon)$ be the maximum cardinality of an $(n,\epsilon)$ separated set. The topological entropy\index{topological entropy} is defined as 
\begin{equation*}
h_{\textup{top}}(f) = \lim_{\epsilon \ra 0} \left( \lim_{n \ra \infty} \sup \frac{1}{n} \log N(n,\epsilon) \right).
\end{equation*}
Next, we make precise the notion of a homeomorphism of a surface permuting a dense collection of domains. 

\begin{defn}\index{permuting a dense collection of domains}\label{defn_perm_dom}
Let $S \subset M$ be compact and $\DD := \{ D_k \}_{k \in \Z}$ the collection of connected components of the complement of $S$,
with the property that $\Int(\Cl(D_k)) = D_k$, where $\Cl(D)$ is the closure of $D$ in $M$. We say $f \in \homeo(M)$ 
{\em permutes a dense collection of domains} if
\begin{enumerate}
\item[(1)] $f(S) = S$ and $\Cl(D_k) \cap \Cl(D_{k'}) = \emptyset$ if $k \neq k'$,
\item[(2)] for every $k \in \Z$, $f^n( D_k) \cap D_k = \emptyset$ for all $n \neq 0$, and
\item[(3)] $\bigcup_{k \in \Z} D_k$ is dense in $M$.
\end{enumerate}
\end{defn}

Note that we do not assume a domain to be recurrent, nor do we assume the orbit of a single domain to be dense. A {\em wandering domain} is a domain with mutually disjoint iterates under $f$ such that the orbit of the domain is recurrent.\index{wandering domain} Thus a diffeomorphism with a wandering domain with dense orbit is a special case of definition~\ref{defn_perm_dom}. Denote $\exp_p \colon T_pM \ra M$ the exponential mapping 
at $p \in M$. The {\em injectivity radius}\index{injectivity radius} at a point $p \in M$ is defined as the largest radius for which $\exp_p$ is 
a diffeomorphism. The injectivity radius $\iota(M)$ of $M$ is the infimum of the injectivity radii over all points $p \in M$. As $M$ is compact, $\iota(M)$ is positive. 

\begin{defn}[Bounded geometry]\label{bounded geometry}
A collection of domains $\{ D_k \}_{n \in \Z}$ on a surface $M$ is said to have {\em bounded geometry} if the following holds: $\Cl(D_k)$ is contractible in $M$ and there exists a constant $\beta \geq 1$ such that for every domain $D_k$ in the collection, there exist $p_k \in D_k$ 
and $0 < r_k \leq R_k$ such that 
\begin{equation}
B(p_k, r_k \subseteq D_k \subseteq B(p_k, R_k) ,~\textup{with}~ R_k / r_k \leq \beta,
\end{equation}
where $B(p,r) \subset M$ is the ball centered at $p \in M$ with radius $r>0$. If no such $\beta$ exists, then the collection is said to 
have {\em unbounded geometry}. 
\end{defn}

By $\Cl(D_k)$ being contractible in $M$ we mean that $\Cl(D_k)$ is contained in an embedded topological disk in $M$. Our definition of bounded geometry 
is equivalent to the notion of bounded geometry in the theory of Kleinian groups and complex dynamics. It is not difficult, given a surface of any genus, to construct homeomorphisms of that surface with positive entropy that permute a dense collection of domains. We show that producing examples that have a certain amount of smoothness is possible only to a limited degree.

\begin{thma}[Topological entropy versus bounded geometry]\label{thm_3a}
Let $M$ be a closed surface and $f \in \diff^{1+\alpha}(M)$, with $\alpha >0$. If $f$ permutes a dense collection of domains with 
bounded geometry, then $f$ has zero topological entropy.
\end{thma}

The outline of the proof of Theorem A is as follows. First we show that the bounded geometry of the permuted domains, combined with their density in the surface, give bounds on the dilatation of $f$ on the complement of the union of the permuted domains. The differentiability assumptions on $f$ allow us to estimate the rate of growth of the dilatation on the whole surface $M$. Using a result by Przytycki~\cite{przy}, we show this rate of growth is slow enough so as to ensure the topological entropy of $f$ is zero.

\section{Entropy and diffeomorphisms with wandering domains}\label{sec_proof_wandering}

First, we study the relation between geometry of domains and the complex dilatation of a diffeomorphism.

\subsection{Geometry of domains and complex dilatation}\label{sub_sec_complex}

We denote $\lambda$ the measure associated to $g$ and $d\lambda$ the Riemannian volume form. By compactness of $M$, there 
exists a constant $\kappa >0$ such that 
\begin{equation}\label{eq_area}
\lambda(B(p,r)) = \int_{B(p,r)} d \lambda \geq \kappa r^2.
\end{equation}
where $B(p,r) \subset M$ is the ball centered at $p$ with radius $r < \iota(M)/2$. A sequence of positive real numbers $x_k$ is called a 
{\em null-sequence}, if for every given $\epsilon>0$ there exist only finitely elements of the sequence for which $x_k \geq \epsilon$. 
Henceforth, we denote $\ell_k := \diam ( D_k )$, the diameter of $D_k$ measured in $g$, with $D_k \in \DD$.

\begin{lem}\label{lem_null}
Let $(M,g)$ be a closed surface and let $\{ D_k \}_{k \in \Z}$ be a collection of mutually disjoint domains with bounded geometry. 
Then the sequence $\ell_k$ is a null-sequence.
\end{lem}

\begin{proof}
Suppose, to the contrary, that $\{ D_k \}_{k \in \Z}$ is not a null-sequence. Then there exist an $\epsilon>0$ and an infinite subsequence 
$k_t$ such that $\diam( D_{k_t} ) \geq \epsilon$. By the bounded geometry property, we have that $\diam( D_{k_t} ) \leq \beta r_{k_t}$ and
therefore $r_{k_t} \geq \epsilon / \beta$. Therefore, by~\eqref{eq_area}, 
\[ \lambda (D_{k_t}) \geq \kappa r_{k_t}^2 \geq \frac{\kappa \epsilon^2}{\beta^2}, \]
for every $t \in \Z$. But this yields that 
\[ \sum_{t \in \Z} \lambda(D_{k_t}) = \infty, \] 
contradicting the fact that $\lambda(M) < \infty$ as $M$ is compact.
\end{proof}

Recall that $S$ is the complement of the union of the permuted domains, i.e. $S = M \setminus \bigcup_{k \in \Z} D_k$. 

\begin{lem}\label{lem_acc}
Let $f \in \homeo(M)$ permute a dense collection $\DD$ of domains with bounded geometry. For every $p \in S$, there exists a sequence of domains $D_{k_t}$ with $\diam(D_{k_t}) \ra 0$ for $t \ra \infty$ such that $D_{k_t} \ra p$.
\end{lem}

\begin{proof}
Fix $p \in S$ and let $U \subset M$ be an open (connected) neighbourhood of $p$. First assume that $p \in S \setminus \bigcup_{k \in \Z} \partial D_k$. 
This set in non-empty, as otherwise the surface $M$ is a union of countably many mutually disjoint continua; but this contradicts Sierpi\'nski's Theorem, which states that no countable union of disjoint continua is connected. We claim that $U$ intersects infinitely many different elements of $\DD$. Indeed, if $U$ intersects only finitely many elements $D_{k_1},...,D_{k_m}$, then $\Omega:= \bigcup_{i=1}^m \Cl(D_{k_i})$ is closed. This implies that $U \setminus \Omega$ is open and non-empty, as otherwise $M$ would be a finite union of disjoint continua, which is impossible. However, as the union of the elements of $\DD$ is dense, $U \setminus \Omega $ can not be open. Thus, there are infinitely many distinct elements $D_{k_1},D_{k_2},...$ of $\DD$ that intersect $U$. Taking a sequence of nested open connected neighbourhoods $U_t$ containing $p$, we can find elements $D_{k_t} \subset U_{t} \setminus U_{t+1}$ for every $t \geq 1$. By Lemma~\ref{lem_null}, $\diam(D_{k_t})$ is a null-sequence and thus we obtain a sequence of domains $D_{k_t}$ with $\diam(D_{k_t}) \ra 0$ for $t \ra \infty$ 
such that $D_{k_t} \ra p$. 

As $\Int(\Cl(D_k)) = D_k$, given $p \in \partial D_k$ and given any neighbourhood $U \ni p$, $U$ has non-empty intersection with 
$M \setminus \Cl(D_k)$. By the same reasoning as above, $p$ is again is a limit point of arbitrarily small domains in the collection $\DD$. Thus we have proved the claim for all points $p \in S$ and this concludes the proof.
\end{proof}

Next, we turn to the {\em complex dilatation}\index{complex dilatation} of a diffeomorphism $f \in \diff(M)$ and its behaviour under 
compositions of diffeomorphisms, see e.g.~\cite{lehto}. We first consider the case where $f \in \diff(\C)$. 
The complex dilatation $\mu_f$ of $f$ is defined by
\begin{equation}
\mu_f \colon \C \ra \D^2, ~ \mu_f(p) = \frac{ f_{\bar{z}} }{ f_z}(p),
\end{equation}
and the corresponding differential 
\begin{equation}
\mu_f(p) \frac{d \bar{z}}{dz},
\end{equation}
is the {\em Beltrami differential}\index{beltrami differential} of $f$. The {\em dilatation}\index{dilatation} of $f$ is defined by 
\begin{equation}\label{eq_defn_dil_1}
K_f(p) = \frac{ 1 + | \mu_f(p) | }{ 1 - | \mu_f(p) | },
\end{equation}
which equals 
\begin{equation}\label{eq_defn_dil_2}
K_f(p) = \frac{ \max_{v} | Df_p(v) | }{ \min_{v} | Df_p(v)| },
\end{equation}
where $v$ ranges over the unit circle in $T_p\C$ and the norm $| \cdot |$ is induced by the standard (conformal) Euclidean metric $g$ 
on $\C$. Denote $[ \cdot, \cdot ]$ be the hyperbolic distance in $\D^2$, i.e. the distance induced by the Poincar\'e metric on $\D^2$. 
When one composes two diffeomorphisms $f,g \colon \C \ra \C$, then 
\begin{equation}\label{eq_comp}
\mu_{g \circ f}(p) = \frac{\mu_f(p) + \theta_f(p) \mu_g(f(p))}{ 1+ \overline{\mu_f(p)} \theta_f(p) \mu_g(f(p))},
\end{equation}
where $\theta_f(p) = \frac{\overline{f_{z}}}{f_{z}}(p)$. It follows that
\begin{equation}\label{eq_comp_iterates}
\mu_{f^{n+1}}(p) = \frac{\mu_f(p) + \theta_f(p) \mu_{f^n}(f(p))}{1+ \overline{\mu_f(p)} \theta_f(p) \mu_{f^n} (f(p))}. 
\end{equation}
We can rewrite~\eqref{eq_comp} as
\begin{equation}\label{eq_comp_iterates_2}
\mu_{g \circ f}(p) = T_{\mu_f(p)} ( \theta_f(p) \mu_{g}(f(p)) )
\end{equation}
where 
\begin{equation}\label{eq_mobius}
T_a(z) = \frac{a+z}{1+ \bar{a} z} \in \Mob(\D^2)
\end{equation}
is an isometry relative to the Poincar\'e metric, for a given $a \in \D^2$. Further, the following relation holds
\begin{equation}\label{eq_dil_mu}
\log ( K_{g \circ f^{-1}} (f(p)) ) = \left[ \mu_g(p), \mu_f(p) \right].
\end{equation}

To define the complex (and maximal) dilatation of a diffeomorphism of a surface $M$, we first lift $f \colon M \ra M$ to the universal cover 
$\widetilde{f} \colon \widetilde{M} \ra \widetilde{M}$ and denote $\pi \colon \widetilde{M} \ra M$ be the corresponding canonical projection mapping,
where $M = \widetilde{M} \slash \Gamma$, with $\Gamma$ a Fuchsian group. We assume here that $\widetilde{M}$ is either $\C$ or $\D^2$, the trivial case of the sphere $\PP^1$ is excluded here. As $\pi$ is an analytic local diffeomorphism, $\widetilde{f}$ is a diffeomorphism. Further, as $M$ is compact, $f$ is $K$-quasiconformal on $M$ for some $K \geq 1$ and thus $\widetilde{f}$ is $K$-quasiconformal on $\widetilde{M}$. 
Since $\widetilde{f} \circ h \circ \widetilde{f}^{-1}$ is conformal for every $h \in \Gamma$, it follows from~\eqref{eq_comp} that
\begin{equation}
\mu_{\widetilde{f}}(p) = \mu_{\widetilde{f}}(h(p)) \frac{\overline{h_z}}{h_z}(p).
\end{equation}
In other words, $\mu_{\widetilde{f}}$ defines a Beltrami differential on $\widetilde{M}$ for the group $\Gamma$, or equivalently, it defines
a Beltrami differential for $f$ on the surface $M$. Furthermore, the same formulas~\eqref{eq_defn_dil_1} and~\eqref{eq_defn_dil_2}, defined relative to the canonical (conformal) metric defined on $M$, hold for the dilatation $K_f$ of $f$ on $M$.

\medskip

The following lemma shows that the bounded geometry assumption of the domains has a strong effect on the dilatation of iterates of $f$ 
on $S$. We say $f$ has {\em uniformly bounded dilatation} on $S \subset M$, if $K_{f^n}(p)$ is bounded by a constant independent 
of $n \in \Z$ and $p \in S$. 

\begin{lem}[Bounded dilatation]\label{lem_bounded_dil}
Let $f \in \diff^1(M)$ permute a dense collection of domains $\DD$. If the collection $\DD$ has bounded geometry, then $f$ has uniformly bounded dilatation on $S$. 
\end{lem}

\begin{proof}
Suppose the collection of domains $\DD = \{D_k\}_{k \in \Z}$ has $\beta$-bounded geometry for some $\beta \geq 1$. Fix $N \in \Z$ and $p \in S$ and take a small open neigbhourhood $U \subset M$ containing $p$. By Lemma~\ref{lem_acc}, there exists a subsequence of domains $D_{k_t}$, where $|k_t| \ra \infty$ and $\diam(D_{k_t}) \ra 0$ for $t \ra \infty$ and such that $D_{k_t} \ra p$. Denote $q = f^N(p) \in S$. We may as well assume that for all $t \geq 1$ the domains $D_{k_t}$ are contained in $U$. Define $D'_{k_t}:=f^N(D_{k_t})$. If we denote $U' = f^N(U)$, then the sequence $D'_{k_{t}}$ converges to $q$ and $D'_{k_{t}} \subset U'$. By the bounded geometry assumption, for every $t \geq 1$, there exists $p_t \in D_{k_t}$ and $0 < r_t \leq R_t$ such that 
\begin{equation*}
B(p_t, r_t) \subseteq D_{k_t} \subseteq B(p_t, R_t) 
\end{equation*}
with $R_k / r_k \leq \beta$. As $f \in \diff^1(M)$, the local behaviour of $f^N$ around $q$ converges to the behaviour of the linear map
$Df^N_{q}$. In particular, if we take $p_t \in D_{k_t}$, then $p_t \ra p$ and thus $q_t :=f^N(q_t) \ra q$, and in order for all $D'_{k_t}$ 
to have $\beta$-bounded geometry, it is required that
\[ K_{f^N}(p)  \leq \frac{R\beta}{r}. \]
Indeed, this is easily seen to hold if the map acts locally by a linear map and is thus sufficient as $f \in \diff^1(M)$ and the 
increasingly smaller domains approach $q$. As $R / r \leq \beta$, we must therefore have $K_{f^N}(p) \leq \beta^2$. As this argument holds for 
every (fixed) $N \in \Z$ and every $p \in S$, we find $\beta^2$ the uniform bound on the dilatation on $S$.
\end{proof}

Our smoothness assumptions on $f$ allow us to give bounds on the (complex) dilatation of iterates of $f$ on $M$ in terms of the 
diameters of the permuted domains. 

\begin{lem}[Sum of diameters]\label{lem_sum_diam}
Let $f \in \diff^{1+\alpha}(M)$, with $\alpha > 0$, which permutes a collection of domains $\DD = \{D_k\}_{k \in \Z}$ with $\beta$-bounded geometry.
Then there exists a constant $C = C(\beta) > 0$ such that, if $p \in D_t$ (for some $t \in \Z$) and $q \in \partial D_t$, then
\begin{equation}\label{eq_sum_diam}
\left[ \mu_{f^{n+1}}(p), \mu_{f^{n+1}}(q) \right] \leq C \cdot \sum_{s=t}^{t+n} \ell_s^{\alpha},
\end{equation}
where the domains are labeled such that $f^s(D_t) = D_{t+s}$.
\end{lem}

To prove Lemma~\ref{lem_sum_diam}, we use the following. 

\begin{lem}\label{lem_prop_1}
Let $f \in \diff^{1}(M)$ and $p_0, q_0 \in M$. Then 
\begin{equation}\label{eq_mobius_1}
\left[ \mu_{f^{n+1}}(p_0), \mu_{f^{n+1}}(q_0) \right] \leq 
\sum_{s=0}^{n} \left[  T_{\mu_f(p_s)} ( \theta_f(p_s) \mu_{f^{n-s}}(q_{s+1}) ), T_{\mu_f(q_s)} ( \theta_f(q_s) \mu_{f^{n-s}}(q_{s+1}) ) \right],
\end{equation} 
where $p_s = f^s(p_0)$ and $q_s = f^s(q_0)$.
\end{lem}

\begin{proof}
Using~\eqref{eq_comp_iterates_2}, we write
\begin{equation*} 
\left[ \mu_{f^{n+1}}(p_0), \mu_{f^{n+1}}(q_0) \right] = 
\left[  T_{\mu_f(p_0)} ( \theta_f(p_0) \mu_{f^n}(p_1) ), T_{\mu_f(q_0)} ( \theta_f(q_0) \mu_{f^n}(q_1) ) \right].
\end{equation*}
By the triangle inequality, we thus have the following inequality
\begin{eqnarray*}
\left[ \mu_{f^{n+1}}(p_0), \mu_{f^{n+1}}(q_0) \right] & \leq &
\left[  T_{\mu_f(p_0)} ( \theta_f(p_0) \mu_{f^n}(p_1) ), T_{\mu_f(p_0)} ( \theta_f(p_0) \mu_{f^n}(q_1) ) \right] \\ &+& 
\left[  T_{\mu_f(p_0)} ( \theta_f(p_0) \mu_{f^n}(q_1) ), T_{\mu_f(q_0)} ( \theta_f(q_0) \mu_{f^n}(q_1) ) \right]. 
\end{eqnarray*}
As both $T_a$ (as defined by~\eqref{eq_mobius}) and rotations are isometries in the Poincar\'e disk, we have that 
\begin{equation*}
\left[ T_{\mu_f(p_0)} ( \theta_f(p_0) \mu_{f^n}(p_1) ), T_{\mu_f(p_0)} ( \theta_f(p_0) \mu_{f^n}(q_1) ) \right]
= \left[ \mu_{f^n}(p_1), \mu_{f^n}(q_1) \right].
\end{equation*}
Inequality~\eqref{eq_mobius_1} now follows by induction.
\end{proof}

As $\partial D_t \subset S$, by Lemma~\ref{lem_bounded_dil}, $\mu_{f^{n-s}}(q_{s+1}) \in B_{\delta}$, with 
$B_{\delta} \subset \D^2$ the compact hyperbolic disk centered at $0 \in \D^2$ with radius 
\begin{equation}
\delta = \frac{ \beta^2-1 }{ \beta^2+1 }.
\end{equation}
Further, define
\begin{equation} 
\delta' = \sup_{p \in M} |\mu_f(p)| < 1,
\end{equation}
and let $B_{\delta'} \subset \D^2$ be the compact hyperbolic disk centered at $0 \in \D^2$ and radius $\delta'$. 

\begin{lem}\label{lem_mobius}
There exists a constant $C_1(\delta, \delta')$ such that 
\begin{equation}\label{eq_mobius_a_b}
\left[ T_{a}(z), T_{b}(z) \right] \leq C_1 \left[ a, b \right],
\end{equation}
for given $a, b \in B_{\delta'}$ and $z \in B_{\delta}$.
\end{lem}

\begin{proof}
First we observe that there exists a constant $0 < \delta'' < 1$ (depending only on $\delta$ and $\delta'$), such that
$[T_{a}(z),0] \leq \delta''$, for every $a \in B_{\delta'}$ and every $z \in B_{\delta}$, as the disks $B_{\delta}, B_{\delta'} \subset \D^2$ are compact. Define $\bar{\delta} = \max\{ \delta, \delta', \delta'' \}$ and $B_{\bar{\delta}} \subset \D^2$ the compact disk with center $0 \in \D^2$ and radius $\bar{\delta}$.

As the Euclidean metric and the hyperbolic metric are equivalent on the compact disk $B_{\bar{\delta}}$, it suffices to show that there exists
a constant $C'_1(\bar{\delta})$ such that
\begin{equation}\label{eq_mobius_eucl}
\left| T_{a}(z) - T_{b}(z) \right| \leq C_1' \left| a - b \right|,
\end{equation}
where $| z - w |$ denotes the Euclidean distance between two points $z,w \in \D^2$. Indeed, if this is shown then~\eqref{eq_mobius_a_b} follows
for a constant $C_1$ which differs from $C_1'$ by a uniform constant depending only on $\bar{\delta}$. To prove~\eqref{eq_mobius_eucl}, we compute 
that
\begin{equation}\label{eq_mobius_lem_1}
\left| T_{a}(z) - T_{b}(z) \right| = \left| \frac{(a-b) + (a \bar{b} - \bar{a} b) z + (\bar{b} - \bar{a}) z^2 }{ (1+ \bar{a} z)( 1+ \bar{b} z)}
\right|.
\end{equation}
As $a, b \in B_{\delta'}$ and $z \in B_{\delta}$, there exists a constant $Q_1(\delta, \delta') >0$ so that 
\[ |(1+ \bar{a} z)( 1+ \bar{b} z)| \geq Q_1. \]
Therefore, it holds that 
\begin{equation}\label{eq_a_b_4}
\left| T_{a}(z) - T_{b}(z) \right| \leq Q_1 \left( |a-b| + \delta' |a \bar{b} - \bar{a} b| + (\delta')^2 |a-b| \right). 
\end{equation}
In order to prove~\eqref{eq_mobius_eucl}, we show there exists a constant $Q_2(\delta') >0$ such that
\begin{equation}\label{eq_a_b_1}
|a \bar{b} - \bar{a} b| \leq Q_2 |a-b|.
\end{equation}
To this end, write $a = r e^{i \phi}$ and $b = r' e^{i \phi'}$ and $x = a \bar{b}$, so that $x = r r' e^{i \nu}$ with $\nu = \phi - \phi'$.
We may assume that $\nu \in [0, \pi)$. It follows that $a \bar{b} - \bar{a} b = x - \bar{x} = 2i r r' \sin(\nu)$. Therefore,
\begin{equation}\label{eq_a_b_2}
|a \bar{b} - \bar{a} b| = |x - \bar{x}| = 2 r r' \sin(\nu) \leq 2 \delta' r \sin(\nu),
\end{equation}
as $r' \leq \delta'$. As the angle between the vectors $a, b \in B_{\delta'}$ is $\nu$, it is easily seen that $|a-b| \geq r \sin(\nu)$. 
Combining this estimate with~\eqref{eq_a_b_2}, we obtain that
\begin{equation}\label{eq_a_b_3}
|a \bar{b} - \bar{a} b| \leq 2 \delta' r \sin(\nu) \leq 2 \delta' |a-b|.
\end{equation}
Setting $Q_2 = 2 \delta'$ yields~\eqref{eq_a_b_1}. If we now combine~\eqref{eq_a_b_3} in turn with~\eqref{eq_a_b_4}, we obtain a uniform
constant 
\[ C_1' (\delta, \delta') = Q_1(1+ \delta' Q_2 + (\delta')^2) = Q_1 ( 1+ 3(\delta')^2) \] 
for which~\eqref{eq_mobius_eucl} holds, as required. 
\end{proof}

\begin{proof}[Proof of Lemma~\ref{lem_sum_diam}]
As $f \in \diff^{1+\alpha}(M)$, we have that $\mu_f(p) \in C^{\alpha}(M, \D^2)$ and $\theta_f \in C^{\alpha}(M,\mathbb{C})$,  
are uniformly H\"older continuous by compactness of $M$. By the triangle inequality, we can estimate the summand in the right-hand side of~\eqref{eq_mobius_1} of Lemma~\ref{lem_prop_1} as
\begin{eqnarray}
& \left[ T_{\mu_f(p_s)} ( \theta_f(p_s) \mu_{f^{n-s}}(q_{s+1}) ), T_{\mu_f(q_s)} ( \theta_f(q_s) \mu_{f^{n-s}}(q_{s+1}) ) \right] \leq & \\ \label{eq_first_term_est}
& \left[ T_{\mu_f(p_s)} ( \theta_f(p_s) \mu_{f^{n-s}}(q_{s+1}) ), T_{\mu_f(q_s)} ( \theta_f(p_s) \mu_{f^{n-s}}(q_{s+1}) ) \right] + & \\ 
\label{eq_second_term_est}
& \left[ T_{\mu_f(q_s)} ( \theta_f(p_s) \mu_{f^{n-s}}(q_{s+1}) ), T_{\mu_f(q_s)} (\theta_f(q_s) \mu_{f^{n-s}}(q_{s+1})) \right]. & 
\end{eqnarray}

To estimate~\eqref{eq_first_term_est}, define 
\[ z_s := \theta_f(p_s) \mu_{f^{n-s}}(q_{s+1}) \in B_{\delta} ~\textup{and} ~ a_s = \mu_f(p_s), b_s = \mu_f(q_s) \in B_{\delta'} \subset \D^2 .\]  
Then~\eqref{eq_first_term_est} reads
\begin{equation}\label{eq_est_mob_1}
\left[ T_{\mu_f(p_s)} ( \theta_f(p_s) \mu_{f^{n-s}}(q_{s+1}) ), T_{\mu_f(q_s)} ( \theta_f(p_s) \mu_{f^{n-s}}(q_{s+1}) ) \right]
= \left[ T_{a_s}(z_s), T_{b_s}(z_s) \right].
\end{equation}
By Lemma~\ref{lem_mobius}, there exists a constant $C_1 >0$ such that
\begin{equation}\label{eq_est_mob_2}
\left[ T_{a_s}(z_s), T_{b_s}(z_s) \right] \leq C_1 [a_s,b_s].
\end{equation}
By H\"older continuity of $\mu_f$, there exists a constant $\widehat{C}_1$ such that 
\begin{equation}\label{eq_est_mob_3}
[a_s, b_s] \leq \widehat{C}_1 (d(p_s, q_s))^{\alpha}.
\end{equation}
Therefore, combining equations~\eqref{eq_est_mob_1},~\eqref{eq_est_mob_2} and~\eqref{eq_est_mob_3}, we obtain that 
\begin{equation}
\left[ T_{\mu_f(p_s)} ( \theta_f(p_s) \mu_{f^{n-s}}(q_{s+1}) ), T_{\mu_f(q_s)} ( \theta_f(p_s) \mu_{f^{n-s}}(q_{s+1}) ) \right]
\leq \widetilde{C}_1 \ell_{t+s}^{\alpha},
\end{equation}
as $d(p_s,q_s) \leq \ell_{t+s}$, with $\widetilde{C}_1 := C_1 \widehat{C}_1$.

To estimate~\eqref{eq_second_term_est}, we note that the hyperbolic distance and the Euclidean distance are equivalent on the compact disk 
$B_{\delta}$. Therefore, as the (Euclidean) distance between a point $z \in B_{\delta}$ and $e^{i \phi} z$ is bounded from above by a constant (depending only on $\delta$) multiplied by the angle $|\phi|$, by H\"older continuity of $\theta_f$ there exists a constant $\widetilde{C}_2(\delta)$, such that 
\[ \left[ \theta_f(p) z , \theta_f(p') z \right] \leq \widetilde{C}_2 (d(p, p'))^{\alpha}, \] 
for all $z \in B_{\delta}$ and $p,p' \in M$, using the local equivalence of the hyperbolic and Euclidean metric. Hence,~\eqref{eq_second_term_est} reduces to 
\begin{equation}
\left[ \theta_f(p_s) \mu_{f^{n-s}}(q_{s+1}), \theta_f(q_s) \mu_{f^{n-s}}(q_{s+1}) \right] \leq \widetilde{C}_2 d ( p_s, q_s ))^{\alpha} 
\leq \widetilde{C}_2 \ell_{t+s}^{\alpha},
\end{equation}
as $d(p_s,q_s) \leq \ell_{t+s}$. Therefore, if we set $C:= \widetilde{C}_1 + \widetilde{C}_2$, then~\eqref{eq_sum_diam} follows.
\end{proof}

\subsection{Upper bounds on the entropy of a surface diffeomorphism}

Next, we relate the topological entropy of a diffeomorphism to its dilatation.

\begin{lem}[Entropy and dilatation]\label{lem_entr_dil}
Let $f \in \diff^{1+\alpha}(M)$ with $\alpha>0$. Then
\begin{equation}\label{eq_entr_dil}
h_{\textup{top}}(f) \leq \lim_{n \ra \infty} \sup \frac{1}{2n} \log \int_M K_{f^n}(p) d \lambda(p),
\end{equation}
with $K_f$ the dilatation of $f$.
\end{lem}

To prove this we use a result of F. Przytycki~\cite{przy}. We need the following notation. Let $L : \R^m \ra \R^m$ be a linear map and $L^{k \wedge} : \R^{m \wedge k} \ra \R^{m \wedge k}$ the induced map on the $k$-th exterior algebra of $\R^m$. $L^{\wedge}$ denotes the induced map on the full exterior algebra. The norm $\| L^{k \wedge} \|$ of $L^k$ has the following geometrical meaning. Let $\vol_k(v_1, ..., v_k)$ be the $k$-dimensional volume of a parallelepiped spanned by the vectors $v_1,..., v_k$, where $v_i \in \R^m$ with $1 \leq i \leq k$. Then
\begin{eqnarray}
\| L^{k \wedge} \| & = & \sup_{v_i \in \R^m} \frac{\vol_k(L(v_1),...,L(v_k))}{ \vol_k(v_1,...,v_k) }, \\
\| L^{\wedge} \| & = & \max_{1\leq k \leq m} \| L^{k \wedge} \|.
\end{eqnarray}
Further, let 
\begin{equation}
\| L \| = \sup_{|v|=1} |L(v)|,
\end{equation}
the standard norm on operators, with $v \in \R^m$ and $| \cdot |$ induced by the corresponding inner product on $\R^m$.
The following result is due to F. Przytycki~\cite{przy} (see also~\cite{oleg}).

\begin{thm}
Given a smooth, closed Riemannian manifold $M$ and $f \in \diff^{1+\alpha}(M)$ with $\alpha>0$. Then
\begin{equation}\label{eq_przy}
h_{\textup{top}}(f) \leq \lim_{n \ra \infty} \sup \frac{1}{n} \log \int_{M} \| (Df^n)^{\wedge} \| d \lambda(p).
\end{equation}
where $h_{\textup{top}}(f)$ is the topological entropy of $f$, $\lambda$ is a Riemannian measure on $M$ induced by a 
given Riemannian metric, ${Df^n}^{\wedge}$ is a mapping between exterior algebras of the tangent spaces 
$T_p M$ and $T_{f^n(p)} M$, induced by the $Df^n_p$ and $\| \cdot \|$ is the norm on operators, induced from the Riemannian metric. 
\end{thm}

\begin{proof}[Proof of Lemma~\ref{lem_entr_dil}]
Fix $p \in M$ and let $Df^n_p : T_p M \ra T_{f^n(p)}M$. Then 
\[ \| Df^n_p \|^2 = K_{f^n}(p) J_{f^n}(p). \] 
Thus 
\begin{equation}
\| (Df^n_p)^{1 \wedge} \| = \sqrt{ K_{f^n}(p) J_{f^n}(p)}, ~\textup{and}~ \| (Df^n_p)^{2 \wedge} \| = J_{f^n}(p).
\end{equation}
It follows that
\begin{equation}
\| (Df^n_p)^{\wedge} \| = \max \left\{ \sqrt{ K_{f^n}(p) J_{f^n}(p)}, J_{f^n}(p) \right\}.
\end{equation}
As 
\begin{equation*}
\max \left\{ \sqrt{ K_{f^n}(p) J_{f^n}(p)}, J_{f^n}(p) \right\} \leq \sqrt{ K_{f^n}(p) J_{f^n}(p)} + J_{f^n}(p),
\end{equation*}
we have that 
\begin{eqnarray*} 
\int_M \| (Df^n_p)^{\wedge} \| d \lambda(p) & \leq & \int_{M} \left( \sqrt{ K_{f^n} J_{f^n}} + J_{f^n} \right) d\lambda \\
& = & \lambda(M) + \int_{M} \sqrt{ K_{f^n} J_{f^n}} d\lambda
\end{eqnarray*}
as $\lambda(M) = \int_{M} J_{f^n} d \lambda$, for every $n \in \Z$. Either $\int_{M} \sqrt{ K_{f^n} J_{f^n}} d\lambda$ is bounded as a sequence in 
$n$, in which case~\eqref{eq_entr_dil} holds trivially, or the sequence is unbounded in $n$, in which case it is readily verified that
\begin{equation*} 
\lim_{n \ra \infty}\sup \frac{1}{n} \log \left( \lambda(M) +  \int_{M} \sqrt{ K_{f^n} J_{f^n}} d\lambda \right) =
\lim_{n \ra \infty}\sup \frac{1}{n} \log \int_{M} \sqrt{ K_{f^n} J_{f^n}} d\lambda.
\end{equation*}
By the Cauchy-Schwartz inequality, we have that
\begin{eqnarray*} 
\int_{M} \sqrt{ K_{f^n} J_{f^n}} d \lambda \leq \sqrt{\lambda(M)} \cdot \sqrt{ \int_{M} K_{f^n} d\lambda}.
\end{eqnarray*}
and thus,
\begin{equation*}
\log \int_{M} \sqrt{ K_{f^n} J_{f^n}} d \lambda \leq \frac{1}{2} \log \lambda(M) + \frac{1}{2} \log \int_{M} K_{f^n} d\lambda.
\end{equation*}
It now follows that 
\begin{equation*}
\lim_{n \ra \infty} \sup \frac{1}{n} \log \int_M \| (Df^n)^{\wedge} \| d \lambda \leq  
\lim_{n \ra \infty} \sup \frac{1}{2n} \log \int_{M} K_{f^n} d\lambda.
\end{equation*}
and this proves~\eqref{eq_entr_dil}.
\end{proof}

\subsection{Proof of Theorem A}\label{subsec_pf_thm3A}

Let us now complete the proof. Let $f \in \diff_A^{1+ \alpha}(M)$, with $\alpha>0$, and suppose that $f$ permutes a dense collection of domains 
$\{ D_k\}_{k \in \Z}$ with bounded geometry. By Lemma~\ref{lem_null}, the sequence $\ell_k$ is a null-sequence. Therefore, $\ell_k^{\alpha}$ is a null-sequence as well, for every $\alpha>0$. Let $p \in D_t$ for some $t \in \Z$ and $q \in \partial D_t$ and label the domains such that 
$f^s(D_t) = D_{t+s}$. By~\eqref{eq_dil_mu}, 
\[ \log K_{f^n}(f(p)) = [ \mu_{f^{n+1}}(p), \mu_{f}(p) ] \] 
and thus, by the triangle inequality, 
\begin{equation}\label{eq_dil_final}
\log K_{f^n}(f(p)) \leq \left[ \mu_{f^{n+1}}(p), \mu_{f^{n+1}}(q) \right] + \left[ \mu_{f^{n+1}}(q), \mu_{f}(p) \right]
\end{equation}
As the second term in the right hand side of~\eqref{eq_dil_final} stays uniformly bounded, we have that
\begin{equation}\label{eq_dil_final_2}
\log K_{f^n}(f(p)) \leq \left[ \mu_{f^{n+1}}(p), \mu_{f^{n+1}}(q) \right] + C'
\end{equation}
for some constant $C'>0$, independent of $p \in M$ and $n \in \Z$. Define 
\[ \xi(n) = \max \sum_{i=0}^n \ell^{\alpha}_{k_i} \]
where the maximum is taken over all collections of $n+1$ distinct elements $\{ D_{k_0}, ..., D_{k_n}\}$ of $\DD$. As $\ell_k^{\alpha}$ 
is a null-sequence, we have that
\begin{equation}\label{eq_to_zero}
\lim_{n \ra \infty} \sup \frac{\xi(n)}{n} = 0.
\end{equation}
By Lemma~\ref{lem_sum_diam}, we have that
\begin{equation*}
\left[ \mu_{f^{n+1}}(p), \mu_{f^{n+1}}(q) \right] \leq C \cdot \sum_{s=t}^{t+n} \ell_s^{\alpha},
\end{equation*}
for some constant $C>0$. Combined with~\eqref{eq_dil_final_2}, we obtain the following uniform estimate 
\begin{equation}
\log K_{f^n}(f(p)) \leq C \xi(n) + C',
\end{equation}
for every $p \in M$ and $n \in \Z$. Therefore 
\begin{eqnarray}
\log \int_M K_{f^n} d\lambda & \leq & \log \int_M \exp(C \xi(n) +C') d \lambda \\
& = & \log \left( ( \exp(C\xi(n) + C' )\lambda(M) \right) \\ \label{eq_last_to_final}
& = & C \xi(n) + C' + \log(\lambda(M)).
\end{eqnarray}

Combining~\eqref{eq_last_to_final} in turn with Lemma~\ref{lem_entr_dil} yields

\begin{equation}\label{eq_final}
h_{\textup{top}}(f) \leq \lim_{n\ra \infty} \sup \frac{1}{2n} \log \int_M K_{f^n} d\lambda \leq C \lim_{n \ra \infty} \sup \frac{\xi(n)}{2n} = 0,
\end{equation}
by~\eqref{eq_to_zero}. This proves Theorem A.

\section{Concluding remarks}

The proof of Theorem A, more precisely condition~\eqref{eq_to_zero} in section~\ref{subsec_pf_thm3A}, fails in the case where the 
H\"older constant $\alpha=0$. This leads to the following natural

\begin{quest}[Differentiable counterexamples]\label{open_counter_entropy}
Do there exist diffeomorphisms $f \in \diff^1(M)$ with positive entropy that permute a dense collection of domains with bounded geometry? 
\end{quest}

\end{document}